\documentclass[12pt]{amsart}

 \usepackage{amsfonts}
\usepackage{amssymb}
 \usepackage{amsthm}

\newtheorem{theorem}{Theorem}[section]
\newtheorem{lemma}[theorem]{Lemma}
\newtheorem{proposition}[theorem]{Proposition}

\theoremstyle{definition}

\theoremstyle{remark}
\newtheorem{assum}{Additional assumption}
\newtheorem{claim}{Claim}
\newtheorem{rem}{Remark}

\numberwithin{equation}{section}

\begin{document}

\title[Commutators in profinite groups]{Commutators and commutator subgroups in profinite groups}
\author{Cristina Acciarri}

\address{Cristina Acciarri:  Department of Mathematics, University of Brasilia,
Brasilia-DF, 70910-900 Brazil}
\email{acciarricristina@yahoo.it}

\author{Pavel Shumyatsky} 

\address{Pavel Shumyatsky: Department of Mathematics, University of Brasilia,
Brasilia-DF, 70910-900 Brazil}

\email{pavel@unb.br}


\thanks{ The first author was supported by FEMAT and CNPq. The second author was supported by CAPES and CNPq.}

\begin{abstract}  Let $G$ be a profinite group. We prove that the commutator subgroup $G'$ is finite-by-procyclic if and only if the set of all commutators of $G$ is contained in a union of countably many procyclic subgroups. 
\end{abstract}

\maketitle
\section{Introduction}
Let $G$ be a profinite group. If $G$ is covered by countably many closed subgroups, then by Baire Category Theorem (\cite[p.\ 200]{kell}) at least one of the subgroups is open. This simple observation suggests that if $G$ is covered by countably many closed subgroups with certain specific properties, then the structure of $G$ is similar to that of the covering subgroups. For example, if $G$ is covered by countably many periodic subgroups, then $G$ is locally finite. We recall that the group $G$ is periodic (or torsion) if each element in $G$ has finite order. The group is locally finite if each of its finitely generated subgroups is finite. Following his solution of the restricted Burnside problem \cite{ze1,ze2} and using Wilson's reduction theorem \cite{Wil}, Zelmanov proved that periodic compact groups are locally  finite \cite{Zel}. Another example is that if $G$ is covered by countably many subgroups of finite rank, then $G$ has finite rank. The profinite group $G$ is said to have finite rank at most $r$ if each closed subgroup of $G$ can be topologically generated by at most $r$ elements. A somewhat less obvious result of the same nature is that a profinite group is covered by countably many procyclic subgroups if and only if it is finite-by-procyclic (see Proposition \ref{cove} in the next section). The group $G$ is called finite-by-procyclic if it has a finite normal subgroup $N$ such that $G/N$ is procyclic.

If $x,y\in G$, then $[x,y]=x^{-1}y^{-1}xy$ is the commutator of $x$ and $y$.
The closed subgroup of $G$ generated by all commutators is the commutator subgroup $G'$ of $G$. In general, elements of $G'$ need not be commutators (see for instance \cite{km} and references therein). On the other hand, Nikolov and Segal showed that for any positive integer $m$ there exists an integer $f(m)$ such that if $G$ is $m$-generator, then every element in $G'$ is a product of at most $f(m)$ commutators \cite{nisega}. Several recent results  indicate that if the set of all commutators is covered by finitely, or countably, many subgroups with certain specific properties, then the structure of $G'$ is somehow similar to that of the covering subgroups.

It was shown in \cite{AS3} that if $G$ is a profinite group that has finitely many periodic subgroups (respectively, subgroups of finite rank) whose union contains all commutators, then $G'$ is locally finite (respectively, $G'$ is of finite rank). In \cite{AS1} similar results were obtained for the case where commutators are covered by countably many subgroups: if $G$ is a profinite group that has countably many periodic subgroups (respectively, subgroups of finite rank) whose union contains all commutators, then $G'$ is locally finite (respectively, $G'$ is of finite rank).
In \cite{DMS} the corresponding results were obtained for profinite groups in which commutators of higher order are covered by countably many periodic subgroups, or subgroups of finite rank.
 It was shown in \cite{FMS} that if $G$ is a profinite group that has finitely many, say $m$, procyclic subgroups whose union contains all commutators, then $G'$ is finite-by-procyclic. In fact, $G'$ has a finite characteristic subgroup $M$ of $m$-bounded order such that $G'/M$ is procyclic. Moreover, if $G$ is a pro-$p$ group that has $m$ procyclic subgroups whose union contains all commutators, then $G'$ is either finite of $m$-bounded order or procyclic. Earlier, Fern\'andez-Alcober and Shumyatsky proved that if $G$ is an abstract group in which the set of all commutators is covered by finitely many cyclic subgroups, then the commutator subgroup $G'$ is either finite or cyclic \cite{FerShu}.

In the present article we deal with profinite groups in which the commutators are covered by countably many procyclic subgroups. The main result is the following theorem.

\begin{theorem}
\label{mainth}
Let $G$ be a profinite group. The commutator subgroup $G'$ is finite-by-procyclic if and only if the set of all commutators of $G$ is contained in a union of countably many procyclic subgroups. 
\end{theorem}

We notice that unlike in the other examples, the structure of pro-$p$ groups $G$ in which the commutators are covered by finitely many procyclic subgroups is different from that of pro-$p$ groups in which commutators are covered by countably many procyclic subgroups: in the former case $G'$ is either finite or procyclic while, according to Theorem \ref{mainth}, this is not necessarily true in the latter case.

An interesting observation that follows from Theorem \ref{mainth} and Proposition \ref{cove} is that if the set of commutators of a profinite group $G$ is contained in a union of countably many procyclic subgroups, then the whole commutator subgroup $G'$ is contained in a union of countably many procyclic subgroups. But of course we cannot claim that the family of  procyclic subgroups that covers the set of all commutators is necessarily the same as the one that covers $G'$.

Another noteworthy fact that can be deduced from Theorem \ref{mainth} concerns profinite groups $G$ such that $G'$ is pro-$p$. Assume that $G$ is such a group and  the set of commutators of $G$ is contained in a union of countably many procyclic subgroups of $G'$.  Then at least one of the subgroups is open in $G'$. 
Indeed, let $C_1,C_2,\dots$ be countably many procyclic subgroups of $G'$ containing the commutators. Theorem \ref{mainth} tells us that $G'$ has a finite normal subgroup $N$ such that $G'/N$ is procyclic. If $G'$ is finite, there is nothing  to prove. So we assume that $G'/N$ is infinite. Since $G'$ is a pro-$p$ group, it follows that any infinite subgroup of $G'$ is open. In particular, if $G$ contains a commutator, say $x\in C_k$, of infinite order, then $C_k$ is open in $G'$. Otherwise, if all commutators in $G$ have finite order, then all of them must belong to $N$ and we conclude that, since $N$ contains all commutators, $N=G'$. In that case $G'$ is finite and the result follows.

We do not know whether the similar phenomenon holds if the co\-ve\-ring subgroups are not necessarily
procyclic.

We have already mentioned that a finite-by-procyclic profinite group is covered by countably many procyclic subgroups. Thus, the hard part of the proof of Theorem \ref{mainth} is the one where we show that if the set of all commutators of $G$ is contained in a union of countably many procyclic subgroups, then $G'$ is finite-by-procyclic. In the next section we provide some helpful results which are used throughout the paper. We also establish that a profinite group is covered by countably many procyclic subgroups if and only if it is finite-by-procyclic (Proposition \ref{cove}). In Sections 3 and 4 we study profinite groups in which the commutators are covered by countably many procyclic subgroups. Section 3 deals with virtually abelian groups while Section 4 is devoted to the metabelian case. After the virtually abelian and the metabelian cases are dealt with, the proof of Theorem \ref{mainth} becomes easy. It is given in the final section.
 
\section{Preliminaries}

For a profinite group $G$ we denote by $\pi(G)$ the set of prime divisors of the orders of finite continuous images of $G$. If a profinite group $G$ has $\pi(G)\subseteq\pi$, then we say that $G$ is a pro-$\pi$ group. Recall that Sylow theorems hold for $p$-Sylow subgroups of a profinite group (see, for example, \cite[Ch.\ 2]{wil-book}). When dealing with profinite groups we consider only continuous homomorphisms and quotients by closed normal subgroups. If $H$ is a closed subgroup of $G$ such that $\pi(H)\subseteq\pi$, we say that $H$ is a pro-$\pi$ subgroup of $G$, or just a $\pi$-subgroup of $G$. The group $G$ possesses a certain property virtually if it has an open subgroup with that property.

\begin{lemma}\label{virtprocyclic}
Let $A$ be an abelian profinite group which is virtually procyclic. Then $A$ is finite-by-procyclic.
\end{lemma} 
\begin{proof}
The lemma is immediate from the fact that a finitely generated profinite  abelian group is a direct sum of finitely many procyclic subgroups (\cite[Theorem 4.3.5]{riza}).
\end{proof}

The next lemma follows from the fact that a direct product of two finite groups of coprime orders is cyclic if and only if both factors are cyclic.
\begin{lemma}
\label{cyclicprod}
Let $A=\prod_{i\in I}A_{i}$ be an abelian profinite group that can be written as a Cartesian product of finite subgroups $A_{i}$ such that $(|A_{i}|,|A_{j}|)=1$ whenever $i\neq j$. Then the following holds:
\begin{itemize}
\item[(i)] $A$ is procyclic if and only if $A_{i}$ is cyclic for each $i$;
 \item[(ii)] $A$ is virtually procyclic if and only if all but finitely many of the subgroups $A_{i}$ are cyclic.
 \end{itemize}
\end{lemma}
We will also require the following lemma taken from \cite[Lemma 2.2]{FMS}.
\begin{lemma}\label{normal-characteristic}
Let $H$ be a characteristic subgroup of a  profinite group $G$. Suppose that $H$ possesses a finite normal subgroup $N$ such that $H/N$ is procyclic. Then $G$ has a finite characteristic subgroup $M$ contained in $H$ such that $H/M$ is procyclic.
\end{lemma}

Given subgroups $A$ and $B$ of a group $G$, we denote by $[A,B]$ the subgroup generated by the set $\{[a,b] \mid a\in A, b\in B\}$.  
\begin{lemma}
\label{commutator}
Let $A$ and $B$ be two subgroups of a group $G$ such that $[A,B]=1$. Suppose that $x$ is a commutator in elements of $A$ and $y$ is a commutator in elements of $B$. Then the element $xy$ is a commutator in $G$.  
\end{lemma}
\begin{proof}
Let $x=[a_{1},a_{2}]$ for some $a_{1},a_{2}\in A$ and $y=[b_{1},b_{2}]$ for some $b_{1},b_{2}\in B$. Since $[A,B]=1$ we have $[a_{1}b_{1},a_{2}b_{2}]=[a_{1},a_{2}][b_{1},b_{2}]=xy$ and the result follows.
\end{proof}
The following result is well-known.
\begin{lemma} 
\label{autocyclic}
Let $G$ be a procyclic group faithfully (and continuously) acted on by a group $A$. Then $A$ is an abelian group.
\end{lemma}

In the above lemma the hypothesis that  $A$ acts on $G$ continuously is superfluous since any automorphism of a procyclic group is continuous.
Now we state a well-known fact about coprime actions on finite groups. As usual, $[G,A,A]$ stands for $[[G,A],A]$.
\begin{lemma}[\cite{gore}, Theorem 3.5.6] 
\label{gorenstein} Let $A$ and $G$ be finite groups with $(|G|,|A|)=1$ and suppose that $A$ acts on $G$. Then we have $[G,A,A]=[G,A].$
\end{lemma}

 We will require the following two results taken from Guralnick \cite{Gur}.
\begin{theorem}[\cite{Gur}, Theorem A]
\label{2-generation}
Let $P$ be a Sylow $p$-subgroup of a finite group $G$. If $P\cap G'$ is abelian and  can be generated by at most two elements, then   $P\cap G'$ consists  entirely  of commutators. 
\end{theorem}  
\begin{lemma}[\cite{Gur}, Lemma 2.5]
\label{gur2} Suppose $G=\langle x,y, B\rangle$ is a finite group with $B$ an abelian subgroup such that  $G'\leq B$ and $[x,y]$ of order $n$. Then the subset $\{[x,y]^{e}b \mid b\in [G,B],(e,n)=1\}$ consists of commutators. 
\end{lemma}
In the present paper Lemma \ref{gur2} will be used in the following special form.
\begin{lemma}\label{Guralnick2}
Let $A=\langle x,y\rangle$ be a finite abelian $2$-generator group acting on a finite abelian group $B$. Then every element of $[B,A]$ is a commutator.
\end{lemma} 
\begin{proof}
Put $G=BA$. It is clear that $G'\leq B$. So it follows from Lemma \ref{gur2}  that $[G,B]$ consists of commutators. In particular, $[B,A]$ consists of commutators.
\end{proof}

\begin{lemma}\label{eleme}
Let $G$ be a group and $N$ a normal subgroup of $G$. Set $K=G/C_G(N)$ and consider the natural action of $K$ on $N$. Then $[N,G]=[N,K]$.
\end{lemma}
\begin{proof} If $n\in N$ and $g\in G$, we have $[n,g]=[n,gC_G(N)]$. Thus, the equality $[N,G]=[N,K]$ follows.
\end{proof}

We denote by $\gamma_\infty(G)$ the intersection of all terms of the lower central series of $G$. It is clear that a finite group $G$ is nilpotent if and only if $\gamma_\infty(G)=1$. Therefore a profinite group $G$ is pronilpotent if and only if $\gamma_\infty(G)=1$. By a well-known property of finite groups $\gamma_\infty(G)$ is generated by all commutators $[x,y]$, where $x$ and $y$ have mutually coprime orders (see for example \cite[Theorem 2.1]{Pavel1}). 
\begin{lemma}\label{fakt}
Let $G=AB$ be a finite group that is a product of two subgroups $A$ and $B$ with $(|A|,|B|)=1$. Suppose that $p_1<p_2$ whenever $p_1\in \pi(A)$ and $p_2\in \pi(B)$. If $\gamma_\infty(G)$ is cyclic, then $B$ is normal in $G$.
\end{lemma}
\begin{proof} Let $G$ be a counterexample of minimal possible order. Then $G$ is not nilpotent and so $\gamma_\infty(G)\neq1$. Thus, since $\gamma_\infty(G)$ is cyclic, it follows that $G$ has a normal subgroup $N$ of prime order $p$, for some prime $p$. By induction, the image of $B$ in  the quotient group $G/N$ is normal and therefore the subgroup $NB$ is normal in $G$. If $p\in\pi(B)$, then $N\leq B$ and we have nothing to prove. Suppose that $p\in\pi(A)$. Since $p$ is smaller than any prime divisor of $|B|$, it follows that $[N,B]=1$,  and so  $NB=N\times B$. Therefore $B$ is characteristic in $NB$, hence normal in $G$ and this leads to a contradiction. 
\end{proof}

We will now prove that a profinite group is covered by countably many procyclic subgroups if and only if it is finite-by-procyclic. The structure of procyclic groups is well-known (cf \cite[Theorem 2.7.2]{riza}). We will use the fact that a procyclic group $G$ such that $|\pi(G)|<\infty$ has only countably many closed subgroups.

\begin{proposition}\label{cove}
A profinite group $G$ is covered by countably many procyclic subgroups if and only if $G$ is finite-by-procyclic.\end{proposition}
\begin{proof}  Suppose first that $G$ has a finite normal subgroup $N$ such that $G/N$ is procyclic. Set $\pi=\pi(N)\cup\pi(\text{Aut}(N))$. Of course, $\pi$ is a finite set of primes. Let $D$ be the subgroup of $G$ generated by all $\pi$-elements and $E$ the subgroup of $G$ generated by all $\pi'$-elements. It is clear that $D=O_\pi(G)$ and $E=O_{\pi'}(G)$. Thus $G=D\times E$.

Let $a\in D$ be an element such that $aN$ is a generator of $D/N$. We know that $\langle a\rangle$ has at most countably many closed subgroups. Let $1=A_1,A_2,\dots$ be the closed subgroups of $\langle a\rangle$. In each $A_i$ we choose a generator $a_i$. Let $B$ be any procyclic subgroup in $D$. There exists $i\geq1$ such that $BN=A_iN$. Clearly, $B=\langle a_ix\rangle$ for a suitable $x\in N$. Since there are at most countably many pairs $(a_i,x)$, it follows that $D$ has only countably many procyclic subgroups, say $D_1,D_2,\dots$. Recall that $G=D\times E$. We now easily deduce that $G$ is covered by countably many procyclic subgroups, each of the form $D_i\times E$. Thus, we proved that if $G$ has a finite normal subgroup $N$ such that $G/N$ is procyclic, then $G$ is covered by countably many procyclic subgroups. Let us now prove the converse. 

Assume that $G$ is covered by countably many procyclic subgroups. We wish to show that $G$ is finite-by-procyclic. By \cite[Theorem 1.1]{Pavel2} $G'$ is finite.  We can pass to the quotient $G/G'$ and without loss of generality assume that $G$ is abelian. Then, by Lemma \ref{virtprocyclic}, $G$ is finite-by-procyclic, as required. The proof is now complete.
\end{proof}

\section{On virtually abelian groups}

 It is clear that a profinite group $G$ has rank one if and only if $G$ is procyclic. It was shown in \cite[Theorem 2]{DMS} that if all commutators in $G$ are contained in a union of countably many subgroups of finite rank, then the rank of $G'$ is finite.   

Let $G$ be, as in Theorem \ref{mainth}, a profinite group in which the set of all commutators is contained in a union of countably many procyclic subgroups. In the course of proving Theorem \ref{mainth} we will often use some simple arguments that show that certain subgroups of $G$ can be assumed procyclic. We will now formalize those arguments as follows.

\begin{rem}
\label{remark0}
Suppose that $G$ has a subgroup $M$ such that every element of $M$ is a commutator.  It follows that the subgroup $M$ is covered by countably many procyclic subgroups. By Baire Category Theorem one of those procyclic subgroups is open. Thus, by Proposition  \ref{cove}, $M$  is finite-by-procyclic. 
\end{rem}

\begin{rem}
\label{remark5}
\noindent Suppose that $G$ has an abelian normal subgroup $A$. For every element $x\in G$ the subgroup $[A,x]$ consists entirely of commutators. Therefore, by Remark \ref{remark0}, the subgroup $[A,x]$  is finite-by-procyclic.
\end{rem}

\begin{rem}
\label{remark2}
\noindent Suppose that $G$ has a normal abelian virtually procyclic subgroup $V$. The set of all torsion elements in $V$ forms a finite characteristic subgroup $M$ such that $V/M$ is procyclic. Then $G'$ is finite-by-procyclic if and only if  the commutator subgroup of $G/M$ is finite-by-procyclic. Thus, we can pass to the quotient $G/M$ and, without loss of generality, assume that $V$ is procyclic.
\end{rem}

\begin{rem}
\label{remark6}
\noindent Let $G$ be a profinite group and $T$ be a procyclic subgroup of $G$. Then $G$ contains a maximal procyclic subgroup $S$ such that $T\leq S$. Indeed, suppose that this is false and write $T=T_{1}<T_{2}<\cdots$, where $T_{i}$ are procyclic subgroups of $G$.  Let $T_{0}$ be the topological closure of $\bigcup_{i}\,T_{i}$.  In any finite quotient of $G$ the image of $T_{0}$ is cyclic and therefore $T_{0}$ is topologically generated by just one element. Hence, $T_{0}$ is procyclic and this proves the claim.
\end{rem}

We will now deal with the following particular case of Theorem \ref{mainth}.
 \begin{lemma}\label{143} Let $G$ be a profinite group in which all commutators are contained in a union of countably many procyclic subgroups. If $G'$ is an abelian pro-$p$ subgroup, then $G'$ is finite-by-procyclic.
\end{lemma}
\begin{proof}  By \cite[Theorem 2]{DMS} $G'$ has finite rank, say $r$. The lemma will be proved by induction on $r$. If $r=1$, then $G'$ is procyclic and there is nothing to prove.  

Assume that $r\geq 2$. Since finitely generated abelian profinite groups decompose as direct sums of procyclic subgroups, torsion elements of $G'$ form a finite normal subgroup, say $G_0$.  We can pass to the quotient $G/G_0$ and, without loss of generality, assume that $G'$ is torsion-free. Since $G'$ is abelian, Remark \ref{remark5} shows that the subgroup $[G',x]$ is finite-by-procyclic for all $x$ in $G$. Since  $G'$ is torsion-free it follows  that $[G',x]$ is infinite procyclic. This happens for every $x\in G$. We also note that since $G$ is metabelian, the subgroup $[G',x]$ is normal in $G$. 

Choose a maximal normal procyclic subgroup $M$ in $G'$. Suppose that $G'/M$ is not torsion-free and let $N/M$ be a finite subgroup  in $G'/M$. Since $G'/M$ is abelian of finite rank, every finite subgroup of $G'/M$ is contained in a finite characteristic subgroup. Hence, we can choose $N$ to be normal in $G$. By Lemma \ref{virtprocyclic} $N$ is finite-by-procyclic. Taking into account that $G'$ is torsion-free, we conclude that $N$ is procyclic. Since $M$ was chosen maximal, this leads to a contradiction. Hence, $G'/M$ is torsion-free. Therefore the rank of $G'/M$ is strictly less than that of $G'$. By induction, $G'/M$ is finite-by-procyclic. We already know that $G'/M$ is torsion-free. Therefore $G'/M$ is procyclic. By a profinite version of Theorem \ref{2-generation} it follows that  every element of $G'$ is a commutator. Hence, by Remark \ref{remark0}, $G'$ is  finite-by-procyclic. 
\end{proof}

As usual, the Frattini subgroup of a group $T$ is denoted by $\Phi(T)$. In the proof of the next  lemma we use the well-known Schur Theorem  that if  $G$ is a group whose center has finite index, then $G'$ is finite (\cite[Theorem 4.12]{Robinson}).  
\begin{lemma}\label{propcase}
 Let $G$ be a profinite group in which all commutators are contained in a union of countably many procyclic subgroups. If $G$ contains an open normal abelian pro-$p$ subgroup $A$, then $G'$ is finite-by-procyclic.
\end{lemma}
\begin{proof}
Suppose that the result is false and take  a counterexample $G$ with the index $[G:A]$ as small as possible. If $[G:A]=1$, there is nothing to prove. If $[G:A]=q$, for some prime $q$, then $G$ is metabelian and the result follows from Lemma \ref{143}. So we assume that $[G:A]$ is not a prime number. 

Assume first that $G/A$ is not simple. Let $K/A$ be a proper non-trivial  normal subgroup of $G/A$. Since $G$ is a counterexample with $[G:A]$ as small as possible, it follows that $K'$ is finite-by-procyclic. Lemma \ref{normal-characteristic} tells us that $K'$ has a finite characteristic subgroup $N$ such that $K'/N$ is procyclic. It is clear that $G'$ is finite-by-procyclic if and only if the commutator subgroup of $G/N$ is finite-by-procyclic. Thus, we can pass to the quotient $G/N$ and, without loss of generality, assume that $K'$ is procyclic. Since the group of automorphisms of a procyclic group is abelian, we have $K'\leq Z(G')$. Since $A$ is open, it follows that the Sylow subgroups of $G$ corresponding to primes other than $p$ are finite. Suppose that $K$ is abelian. We pass to the quotient over the (finite) subgroup generated by all $p'$-elements of $K$ and assume that $K$ is a pro-$p$ group. Since the index $[G:K]$ is smaller than $[G:A]$, by induction we deduce that $G'$ is finite-by-procyclic. Thus, in the case where $K$ is abelian, $G'$ is finite-by-procyclic as required. This shows that $G'/K'$ is finite-by-procyclic.
Let $L$ be the minimal normal subgroup of $G'$  such that $L/K'$ is finite and $G'/L$ is procyclic. Since $K'$ is contained in $Z(G')$, it follows from the Schur Theorem that $L'$ is finite. We pass to the quotient over  $L'$ and assume that $L$ is abelian. Further, the argument in Remark \ref{remark2} allows us to assume that $L$ is procyclic, in which case $L\leq Z(G')$. Since also $G'/L$ is procyclic, it follows that  $G'$ is abelian. Now the result is immediate from Lemma \ref{143}. Thus, in the case where $G/A$ is not simple, or simple of prime order, we are done.

Assume  that $G/A$ is a non-abelian simple group. Of course, in this case we have $[A,G,G]=[A,G]$. Since $A$ is normal abelian, we apply Remarks \ref{remark5} and \ref{remark2} and assume that $[A,x]$ is  procyclic for every $x\in G$.  Here we use the fact that $G$ contains only finitely many subgroups of the form $[A,x]$. Put $M=[A,G]$. Suppose first that $M$ is procyclic. Then the equality $M=[M,G]$ implies that $M=1$. Therefore $G$ is central-by-finite and, by Schur's theorem $G'$ is finite. Thus, we may  assume that $M$ is of rank $n\geq 2$. Since $G/A$ acts faithfully on $M/\Phi(M)$, we have an embedding of $G/A$ in $GL(V)$, where $V=M/\Phi(M)$. However, for any $p'$-element $g$ of $SL(V)$ the dimension of $[M,g]$ must be at least two. This leads to a contradiction since $[A,x]$ is  procyclic for every $x\in G$. The proof is now complete. 
\end{proof}

We will now look at the case where the open normal abelian subgroup $A$ is not necessarily a pro-$p$ subgroup. Our immediate  goal is to show that $[A,G]$ is finite-by-procyclic. We denote by $O_{\pi}(G)$ the unique largest normal pro-$\pi$ subgroup of $G$. 
\begin{lemma}\label{AcomaG} 
Let $G$ be a profinite group in which all commutators are contained in a union of countably many procyclic subgroups. If $G$ contains an open normal abelian subgroup $A$, then $[A,G]$ is finite-by-procyclic.
\end{lemma}

\begin{proof} Choose a prime $p\in\pi(A)$ and let $P$ be the Sylow pro-$p$ subgroup of $A$. Note that $[P,G]$ is the Sylow pro-$p$ subgroup of $[A,G]$.  Moreover, by Lemma \ref{eleme}, in the semidirect product of $P$ by $G/C_G(P)$ we have $[P,G]=[P,G/C_{G}(P)]$. Since $A$ is open in $G$, the semidirect product of $P$ by $G/C_G(P)$ is a virtually pro-$p$ group and it follows from Lemma \ref{propcase} that $[P,G]$ is finite-by-procyclic. It is now straightforward that for any finite set of primes $\sigma$ the subgroup $O_\sigma([A,G])$ is finite-by-procyclic. 

Let $\pi=\pi(G/A)$. Since $\pi$ is finite, it follows that $O_\pi([A,G])$ is finite-by-procyclic. We observe that $[A,G]=O_{\pi}([A,G])\times O_{\pi'}([A,G])$. Therefore it is sufficient to show that $O_{\pi'}([A,G])$ is finite-by-procyclic. We can pass to the quotient $G/O_{\pi}([A,G])$ and simply assume that $[A,G]=O_{\pi'}([A,G])$.

Let $K=G/C_G(A)$. Thus, $K$ is a finite group acting on $A$ by automorphisms.
Applying the argument in Remarks \ref{remark5} and \ref{remark2} to all subgroups of the form $[A,g]$ (there are only finitely many of them) we assume that $[A,g]$ is procyclic for any $g\in K$. Thus, a finite $\pi$-group $K$ acts faithfully and continuously on the profinite abelian $\pi'$-group $A$ in such a way that $[A,g]$ is procyclic for any $g\in K$. Suppose that $[A,K]$ is not procyclic. Then we can choose a Sylow pro-$p$ subgroup $P\leq[A,K]$ such that $P$ is not procyclic and $P=[P,K]$. By \cite[Lemma 2.11]{AS2} either the quotient group  $K/C_K(P)$ is cyclic, or otherwise, there exists an element $g_{0}\in K$ such that $[P,g_{0}]$ is not procyclic. Obviously under our assumptions the latter case is impossible so we conclude that $K/C_K(P)$ is cyclic. Write $K=\langle g,C_K(P)\rangle$ for some element $g\in K$. Then we have $P=[P,K]=[P,g]$. Since $[P,g]$ is procyclic, this is a contradiction. 
\end{proof}

Now we are ready  to prove the main result of this section.
\begin{theorem}\label{virtabeliancase}
Let $G$ be a profinite group in which all commutators are contained in a union of countably many procyclic subgroups.  If $G$ contains an open normal abelian subgroup $A$, then $G'$ is finite-by-procyclic.
\end{theorem}
\begin{proof}  We know from Lemma \ref{AcomaG} that $[A,G]$ is finite-by-procyclic. We use Remark \ref{remark2} and assume that $[A,G]$ is procyclic.  It follows from  Lemma \ref{autocyclic} that $[A,G]\leq Z(G')$. In the quotient group $G/[A,G]$ the subgroup $A$ is central so, by Schur's theorem, $G'/[A,G]$ is finite and hence $G'$ is central-by-finite. In particular, $G''$ is finite. We pass to the quotient over $G''$ and assume that $G'$ is abelian. Since $G'$ is virtually procyclic, it follows from Lemma \ref{virtprocyclic} that $G'$ is finite-by-procyclic. The proof is complete.
\end{proof}

\section{On metabelian groups} 

The purpose of the present section is to prove the following theorem.

\begin{theorem}\label{meta}
If $G$ is a  metabelian profinite group in which all commutators are contained in a union of countably many procyclic subgroups, then $G'$ is finite-by-procyclic.
\end{theorem}

The special case where $G'$ is an abelian pro-$p$ subgroup was already proved in Lemma \ref{143}. The general case, where $G'$ is not necessarily pro-$p$,  is more complicated since it does not reduce easily to the situation where $G'$ is torsion-free. The case of the theorem where $G'$ is torsion-free is quite easy. Indeed,  for any $p\in\pi(G')$ we consider the quotient group $G/O_{p'}(G')$. In view of Lemma \ref{143} we conclude that each Sylow $p$-subgroup of $G'$ is procyclic. Hence, in the case where $G'$ is torsion-free $G'$ is procyclic. 

\begin{proof}[Proof of Theorem \ref{meta}] We start with an easy observation that any Sylow subgroup of $G'$ is finite-by-procyclic. Indeed, choose a prime $p\in \pi(G')$ and consider the quotient $G/O_{p'}(G')$. The commutator subgroup of $G/O_{p'}(G')$ is a pro-$p$ group and therefore, by Lemma \ref{143}, it is finite-by-procyclic. Thus, the Sylow $p$-subgroup of $G'$ is finite-by-procyclic.

Let $C_1,C_2,\dots,$ be the countably many procyclic subgroups of $G'$ whose union contains all commutators. Let $x\in G$. By Remark \ref{remark5} $[G',x]$ is virtually procyclic and consists of commutators. Therefore for any $x\in G$ there exist positive numbers $n(x)$ and $i(x)$ such that $[G',x]^{n(x)}\leq C_{i(x)}$. 

For each pair $\alpha=(n,i)$ we define the set $$S_{\alpha}=\{x\in G \mid [G',x]^n\leq C_i\}.$$ The sets $S_{\alpha}$ are closed in $G$. Indeed, fix a pair $\alpha=(n,i)$ and suppose that $x\not\in S_{\alpha}$. It follows that $[G',x]$ contains an element $y$ such that $y^n\not\in C_i$. We can choose an  open normal subgroup $N$ in $G$ such that the image of $y^n$ in $G/N$ is not contained in the image of $C_i$. We see that no element that belongs to the coset $xN$ is contained in $S_{\alpha}$. Therefore the complement of $S_{\alpha}$ is open (for each element in the complement there exists a neighborhood of that element which is entirely contained in the complement). Therefore the set $S_{\alpha}$ is closed.

It is clear that the group $G$ is covered by the sets $S_{\alpha}$. By Baire's Category Theorem at least one of these sets contains a non-empty interior. Thus, there exists a certain pair $(\overline{n},\overline{\imath})$ such that $G$ possesses an element $b$ and an open normal subgroup $H$ with the property that $[G',x]^{\overline{n}}\leq C_{\overline{\imath}}$ for any $x\in bH$. Let $\pi^{*}$ be the set of prime divisors of $\overline{n}$. Since $G'$ is abelian of finite rank (\cite[Theorem 2]{DMS}), the subgroup generated by all $\pi^{*}$-elements of finite order in $G'$ is finite. Passing to the quotient over this subgroup we can assume that for any prime $p\in\pi^{*}$ the Sylow $p$-subgroup of $G'$ is infinite procyclic. If necessary, we enlarge the subgroup $C_{\overline{\imath}}$ replacing it by the product of $C_{\overline{\imath}}$ with all Sylow subgroups of $G'$ corresponding to the primes in $\pi^{*}$. Obviously this product is again a procyclic subgroup. To avoid changing the notation  we simply assume that $C_{\overline{\imath}}$ contains the Sylow subgroups  corresponding to the primes in $\pi^{*}$. It follows that $[G',x]\leq C_{\overline{\imath}}$ for any $x\in bH$. Taking into account that each subgroup of the form $[G',x]$ is normal in $G$ we deduce that $[G',h]\leq [G',bh][G',b]\leq C_{\overline{\imath}}$ for any $h\in H$. Hence, $[G',H]\leq C_{\overline{\imath}}$. Let us prove the following claim. 

\begin{claim} \label{gamma3claim} The third term of the lower central series of $G$ (denoted by $\gamma_{3}(G)$) is virtually procyclic. \end{claim}

Let  $\pi(G)=\{p_1,p_2,\dots\}$ and $\{G_1,G_2,\dots\}$ be a Sylow system in $G$ such that $G_i$ is a Sylow $p_i$-subgroup of $G$. Thus, $G_iG_j=G_jG_i$ for all $i,j$. For every $i=1,2\dots$ put $R_i=\gamma_3(G)\cap G_i$.  Since $\Phi(\gamma_{3}(G))=\prod_{i}\Phi(R_{i})$, it follows from Lemma \ref{cyclicprod}(ii) that $\gamma_{3}(G)/\Phi(\gamma_{3}(G))$ is virtually procyclic if and only if $R_{i}/\Phi(R_{i})$ is cyclic for all but finitely many primes $p_{i}$. Therefore, since the Sylow subgroups of $G'$ are finite-by-procyclic, if $\gamma_{3}(G)/\Phi(\gamma_{3}(G))$ is virtually procyclic then so is $\gamma_{3}(G)$.

Assume by absurdum that $\gamma_3(G)$ is not virtually procyclic. Pass to the quotient over $\Phi(\gamma_{3}(G))$ and assume that every subgroup $R_i$ is elementary abelian. Further, each subgroup $R_i$ is finite since $G'$ has finite rank. 

Let $G=\langle H,b_1,\dots,b_s\rangle$. Set $K_0=H$, $K_1=\langle H,b_1\rangle$, $\dots$, $K_s=\langle H,b_1,\dots,b_s\rangle=G$. Let $\widehat{s}$ be the minimal index in $\{0,\ldots,s\}$ for which $[G',K_{\widehat{s}}]$ is not virtually procyclic. Since $[G',K_{0}]=[G',H]\leq C_{\overline{\imath}}$, it is clear that $1\leq \widehat{s}\leq s$. Since all subgroups $[G',b_j]$ consist of commutators, they are virtually procyclic. Using Remark~\ref{remark2} we assume that each subgroup $[G',b_j]$ is procyclic. For the same reason we can assume that $[G',K_{\widehat{s}-1}]$ is procyclic. 

Let $\sigma$ be the set of all primes for which the corresponding Sylow subgroups of $[G',K_{\widehat{s}}]$ are not cyclic. Since $[G',K_{\widehat{s}}]$ is not virtually procyclic, it follows that the set $\sigma$ is infinite. Moreover  $\gamma_3(G/O_{\sigma'}(\gamma_3(G)))$ is not virtually procyclic. Consider the quotient $G/O_{\sigma'}(\gamma_3(G))$ in place of $G$ and just assume that $\sigma=\pi(\gamma_3(G))$. Since $[G',K_{\widehat{s}}]=[G',K_{\widehat{s}-1}][G',b_{\widehat{s}}]$ and since both subgroups $[G',K_{\widehat{s}-1}]$ and $[G',b_{\widehat{s}}]$ are procyclic while none of the Sylow subgroups of $[G',K_{\widehat{s}}]$ is cyclic, we conclude that $[G',K_{\widehat{s}-1}]\cap [G',b_{\widehat{s}}]=1$ and $\pi([G',K_{\widehat{s}-1}])=\pi([G',b_{\widehat{s}}])=\sigma$.

Suppose that $[G',h]$ is finite for any $h\in H$.  For any positive integer $\lambda$ define $$S_{\lambda}=\{h\in H \mid |[G',h]|\leq \lambda\}.$$ The sets $S_{\lambda}$ cover the subgroup $H$ and it is clear that each set $S_{\lambda}$ is closed. By Baire's Category Theorem at least one of these sets contains a non-empty interior. Hence there exists an integer $m$ such that $H$ contains a open normal subgroup $H_1$ with the property that $|[G',h]|\leq m$ for any $h\in H_{1}$. Since $G'$ is abelian of finite rank, the subgroup generated by all elements of finite order at most $m$ in $G'$ is finite.
We pass to the quotient over this subgroup and thus assume that $H_{1}\leq C_{G}(G')$. In that case $G/C_G(G')$ is finite. The semidirect product of $G'$ by $G/C_G(G')$ is virtually abelian. Lemma \ref{eleme} tells us that $[G',G]=[G',G/C_{G}(G')]$. By Theorem \ref{virtabeliancase} the commutator subgroup of the semidirect product of $G'$ by $G/C_G(G')$ is finite-by-procyclic. Since $\gamma_{3}(G)=[G',G/C_{G}(G')]$, we now conclude that $\gamma_{3}(G)$ is finite-by-procyclic. This is a  contradiction since we have assumed that $\gamma_{3}(G)$ is not virtually procyclic. Hence there exists $h_{0}\in H$ such that $[G',h_{0}]$ is infinite. 

Set $D=\langle b_{\widehat{s}},h_{0}\rangle$. Lemma \ref{eleme} shows that $[G',D]=[G',D/C_{D}(G')]$. Furthermore $D'$ centralizes $G'$ and so $D/C_{D}(G')$ is an abelian $2$-generator group. By Lemma \ref{Guralnick2}, we deduce that every element of $[G',D]$ is a commutator. By Remark \ref{remark0} we know that $[G',D]$ is virtually procyclic. An application of Remark~\ref{remark2} allows us to assume that $[G',D]$ is procyclic.

Recall that $[G',b_{\widehat{s}}]\leq[G',D]$. Since $\pi([G',b_{\widehat{s}}])=\sigma$, it follows that $[G',b_{\widehat{s}}]=[G',D]$. In particular $[G',h_{0}]\leq[G',b_{\widehat{s}}]$. Taking into account that $h_0\in K_{\widehat{s}-1}$, we deduce that the infinite subgroup $[G',h_{0}]$ is contained in $[G',K_{\widehat{s}-1}]$. Therefore $[G',h_{0}]$ is contained in the intersection of $[G',K_{\widehat{s}-1}]$ and $[G',b_{\widehat{s}}]$. We have already remarked that the intersection is trivial. It follows that $[G',h_0]=1$. This is a contradiction. The proof of Claim \ref{gamma3claim} is complete.

Now using Remark \ref{remark2} we may take the following assumption.  
\begin{assum}
\label{assum1}
$\gamma_3(G)$ is procyclic. 
\end{assum}
Recall that  $\{G_1,G_2,\dots\}$ is a Sylow system in $G$ such that $G_i$ is a Sylow $p_i$-subgroup of $G$ and $G_iG_j=G_jG_i$. From now on we assume that $p_i<p_j$, whenever $i<j$. By a profinite version of  Lemma \ref{fakt} $G_i$ normalizes $G_j$, whenever $i<j$. For every $i=1,2\dots$ set $P_i=G'\cap G_i$.  Since $\Phi(G')=\prod_{i}\Phi(P_{i})$, it follows from Lemma \ref{cyclicprod}(ii) that $G'/\Phi(G')$ is virtually procyclic if, and only if, $P_{i}/\Phi(P_{i})$ is cyclic for all but finitely many primes $p_{i}$. Therefore, since the Sylow subgroups of $G'$ are finite-by-procyclic, $G'/\Phi(G')$ is virtually procyclic if and only if $G'$ is virtually procyclic. Hence, we can pass to the quotient over  $\Phi(G')$ and  assume that each subgroup $P_i$ is elementary abelian. Thus, from now on our task is to establish the following fact. 
\begin{claim}
\label{claimPi}
 All but finitely many subgroups $P_i$ are cyclic. 
 \end{claim} 
 Suppose that this is false.  Each  subgroup  $P_{i}$ is finite since $G'$ is of finite rank. Whenever the rank of $P_i$ is at least three we choose a subgroup $M_i\leq P_i$ which is normal in $G$ and satisfies the condition that $P_i/M_i$ has rank two. Such a subgroup $M_{i}$ does exist because, by  Additional assumption \ref{assum1}, $\gamma_{3}(G)$ is procyclic. Observe that after  passing to the quotient over the Cartesian product of all such $M_i$, we still have a counterexample to Claim \ref{claimPi}. Thus  we can make the following  
 \begin{assum}
 \label{assum2}
 The rank of each $P_i$ is at most two. 
 \end{assum} 
By Theorem \ref{2-generation}, every element in each $P_i$ is a commutator.  However we cannot claim that $G'$ consists of commutators. 

\begin{claim}\label{pronilpo}
If $G$ is pronilpotent, then $G'$ is finite-by-procyclic.
\end{claim}
Indeed, if $G$ is pronilpotent, then $G$ is the Cartesian product of  the subgroups $G_i$ and so $G'$ is the Cartesian product of the subgroups $P_i$. Combining the fact that  every element in each $P_i$ is a commutator in elements of $G_i$ with Lemma \ref{commutator}, we deduce that every element of $G'$ is a commutator. Thus, by Remark \ref{remark0}, $G'$  is finite-by-procyclic, as claimed. This completes the proof of Claim \ref{pronilpo}.

Recall that $\gamma_{\infty}(G)$ stands for the intersection of all terms of the lower central series of $G$. We know from Claim \ref{pronilpo}  that the commutator subgroup of $G/\gamma_\infty(G)$ is finite-by-procyclic. Hence, if $\gamma_{\infty}(G)$ is finite, then $G'$ is finite-by-procyclic. Therefore, without loss of generality, we may make the following  
\begin{assum} 
\label{assum3} The subgroup $\gamma_\infty(G)$ is infinite.
\end{assum}

By Claim \ref{pronilpo}  the quotient $G'/\gamma_\infty(G)$ is finite-by-procyclic. Let $T$ be the largest subgroup in $G$ such that $\gamma_\infty(G)\leq T\leq G'$ and $\gamma_\infty(G)$ is open in $T$. In view of Remark \ref{remark2} we can assume that $T$ is procyclic. Let $\tau$ be the set $\pi(T)\setminus\pi(\gamma_\infty(G))$. Then $O_\tau(G')$ is a finite characteristic subgroup in $G$. We pass to the quotient $G/O_\tau(G')$ and simply assume that $T=\gamma_\infty(G)$. Thus, we conclude that $G'/\gamma_\infty(G)$ is procyclic. Since every Sylow subgroup $P_i$ is elementary abelian of rank at most two, whenever $P_i$ is non-cyclic we have $P_i\cap\gamma_\infty(G)\neq1$.

For every $i\geq 2$ set $H_i=\prod_{j<i}G_j$. We know that $H_i$ normalizes $G_i$ and it is clear that $p_i\not\in\pi(H_i)$.  Observe that $P_{i}=[G_{i}, H_{i}G_{i}]$. By a profinite version of Theorem 5.3.5 of \cite{gore}, we have $G_i=C_{G_i}(H_i)[G_i,H_i]$. Further, since $p_i\not\in\pi(H_i)$, it follows that $[G_i,H_i]$ is contained in $\gamma_\infty(G)$. The latter is procyclic and therefore $[G_i,H_i]$ is of prime order $p_i$. We now deduce that $G_i=C_{G_i}(H_i)\times [G_i,H_i]$. Note that whenever $P_i$ is non-cyclic the subgroup $C_{G_i}(H_i)$ must be non-abelian. Otherwise we would have $P_i=[G_i,H_i]\leq \gamma_{3}(G)$, which leads to a contradiction with Additional Assumption \ref{assum1}. Thus, for any $i$ such that $P_i$ is non-cyclic there exist non-commuting elements $a_i,b_i\in C_{G_i}(H_i)$. 

Let $X$ be the set of all commutators in $G$.  The set $X$ is closed (\cite[Ex. 6, Chap.\ 1]{analytic}) and it is equipped with the topology  inherited from $G$. The set $X$ is covered by countably many closed subsets $C_i\cap X$ and,  by Baire's Category Theorem, at least one of those subsets contains a non-empty interior. Thus  there exist a positive integer $\overline{\jmath}$, an open subgroup $N\leq G$ and $x\in X$, with $x=[a,b]$ for some elements $a,b\in G$, such that $X\cap xN\leq C_{\overline{\jmath}}$. Of course, the subgroup $N$ can be taken normal. Thus all commutators contained in the coset $[a,b]N$ lie in $C_{\overline{\jmath}}$.

Let us denote by $L$ the product of all $G_i$ for which $p_i\leq|G/N|$ and by $J$ the product of all $G_i$ for which $p_i>|G/N|$.  Since  $G_i$ normalizes $G_j$, whenever $i<j$, it follows that $J$ is normalized by $L$. Moreover since $J$ is the product of Sylow $p_{i}$-subgroups $G_i$ for which $p_i$ does not divide $|G/N|$, we have   $J\leq N$.  Since $G=LN$, without loss of generality, we can assume that  $x=[a,b]$, with $a,b\in L$.

Set $$I=\{ i \mid  p_i>|G/N|\, \text{and}\, P_i\, \text{is not cyclic}\}.$$ We can assume that  the set $I$ is infinite, since otherwise  we would have all but finitely many $P_{i}$ cyclic, as required. 

Note that the subgroup $\langle a,b\rangle$ commutes with the subgroup $\langle a_i,b_i\rangle$ for any non-commuting elements $a_i,b_i\in C_{G_i}(H_i)$ and any $i\in I$. This is because $a_{i}$ and $b_{i}$ are taken in $C_{G_{i}}(H_{i})$, and so in particular they centralize $a,b$ which belong to $L\leq H_i$. It follows from Lemma \ref{commutator}  that the elements of the form $[a,b][a_i,b_i]$ are commutators and so they all lie in $C_{\overline{\jmath}}$. In particular, all commutators $[a_i,b_i]$ lie in $C_{\overline{\jmath}}$, since so does $[a,b]$. 

For any $g\in G$ we have $[a,gb]=[a,b][a,g]^b$. Hence all elements of the form $[a,b][a,g]^b$ are commutators. Since all commutators contained in the coset $[a,b]N$ lie in $C_{\overline{\jmath}}$, it follows that $[a,N]^b\leq C_{\overline{\jmath}}$.

 Now suppose that $g\in G_i$ for some $i\in I$. Both $a$ and $b$ normalize $G_{i}$ since $L$ does. On the other hand, $a$ and $g$ are elements of coprime orders and therefore $[a,g]^{b}$ lies in $\gamma_{\infty}(G)$. It follows that $[a,g]^b\in\gamma_3(G)\cap G_i$. If $[a,g]\neq1$, then $P_{i}=\langle[a,g],[a_i,b_i]\rangle$, where again $a_i$ and $b_i$ are arbitrary non-commuting elements from $C_{G_i}(H_i)$. Indeed, both $[a,g]$ and $[a_{i},b_{i}]$ belong to $P_{i}$. Observe that they cannot be in the same cyclic subgroup of $P_{i}$. This is because  $[a_{i},b_{i}]\in C_{G_{i}}(H_{i})$ while $[a,g]\in [G_{i},H_{i}]$ and we know that $G_{i}=C_{G_{i}}(H_{i})\times [G_{i},H_{i}]$.  By Additional assumption \ref{assum2} the subgroup $P_{i}$ is of rank at most 2. Due to the choice of $i$ the subgroup $P_{i}$ is non-cyclic. Therefore  $P_{i}=\langle[a,g],[a_i,b_i]\rangle$, as desired. However this is a contradiction since both $[a,g],[a_i,b_i]$ belong to $C_{\overline{\jmath}}$. Thus, we conclude that $a$ commutes with $G_i$ whenever  $i\in I$. 

Similarly, for any $g\in G$ we have $[ga,b]=[g,b]^{a}[a,b]$. Hence all elements of the form $[g,b]^{a}[a,b]$ are commutators and we have $[N,b]^{a}\leq C_{\overline{\jmath}}$. Thus, arguing as in the preceding paragraph we deduce that $b$ commutes with $G_i$ whenever $i\in I$.
 
For any $c\in C_L(a)$ we have $[a,cb]=[a,b]$. Thus, taking in the above argument the element $cb$ in place of $b$, we conclude that $cb$ commutes with $G_i$ whenever $i \in I$. This holds for each $c\in C_L(a)$ and therefore $C_L(a)$ commutes with $G_i$ whenever $i\in I$. Observe that $L'$ is finite, because it is the product of  finitely many  finite Sylow subgroups of $G'$. It follows that $C_L(a)$ has finite index in $L$ and so, being closed, $C_L(a)$ must be open in $L$. Now it is easy to see that $L$ contains an open  normal subgroup $K$ which commutes with $G_i$ whenever $i\in I$. 

Now let the indices $i,j$ satisfy the condition that $|G/N|<p_i<p_j$. Then $G_iG_j\leq J$ and $G_i$ normalizes $G_j$. Assume additionally that $P_j$ is not cyclic and suppose that $[G_i,G_j]\neq1$. If $y\in[G_i,G_j]$,  then $y$ can be written as a commutator $[y_1,y_2]$, with $y_1\in G_i$ and  $y_2\in[G_i,G_j]$. This is because by Lemma  \ref{gorenstein} we have  $[G_{j},G_{i},G_{i}]=[G_{j},G_{i}]$. Recall that $G_i=C_{G_i}(H_i)\times~ [G_i,H_i]$. Since $[G_i,H_i]$ is contained in $O_{p_i}(G)$, it follows that $[G_i,H_i]$ commutes with $G_j$. Therefore the element $y_1$ can be chosen in $C_{G_i}(H_i)$. We deduce that the subgroup $\langle a,b\rangle$ commutes  with the subgroup $\langle y_1,y_2\rangle$. Indeed, on the one hand, $y_{1}\in C_{G_{i}}(H_{i})$ and so it centralizes $a$ and $b$. On the other hand, it has been shown above that both $a$ and $b$ commute with all element of $G_{j}$ whenever $j\in I$. In particular both $a$ and $b$ commute with $y_{2}$. It follows from Lemma \ref{commutator} that the element $[a,b]y$ is a commutator. In view of the fact that all commutators contained in the coset $[a,b]N$ lie in $C_{\overline{\jmath}}$, we conclude that $y\in C_{\overline{\jmath}}$. If $y\neq1$, then $P_j=\langle[a_j,b_j],y\rangle$. Indeed both $[a_{j},b_{j}]$ and $[y_{1},y_{2}]$ lie in $P_{j}$. These elements do not belong to the same cyclic subgroup since $[a_{j},b_{j}]$ is in $C_{G_{j}}(H_{j})$ and $y=[y_{1},y_{2}]$ lies in $[G_{j},H_{j}]$ (we use here that $y_{1}\in H_{j}$ and $y_{2}\in G_{j}$). Moreover by Additional assumption \ref{assum2} the subgroup $P_{j}$ is of rank at most 2 and we are assuming that $P_{j}$ is non-cyclic. Therefore $P_{j}=\langle[a_j,b_j],y\rangle$. This leads to a contradiction since we know that both $[a_j,b_j]$ and $y$ lie in $C_{\overline{\jmath}}$. Thus, we have $[G_i,G_j]=1$ whenever the indices $i$ and $j$ satisfy the conditions that $|G/N|<p_i<p_j$ and $P_j$ is non-cyclic.

Let us show that for every $j\in I$ there exists $i<j$ such that $[G_i,G_j]\not=1$. Suppose that $[G_i,G_j]=1$ for all $i<j$ and put $q=p_{j}$. Then $G/O_{q'}(G')$ is pronilpotent. It follows that $\gamma_{\infty}(G)\leq O_{q'}(G')$. This contradicts the assumption that $P_{j}\cap\gamma_{\infty}(G)\neq 1$ whenever $P_{j}$ is noncyclic. Thus, indeed for every $j\in I$ there exists $i<j$ such that $[G_i,G_j]\not=1$.

On the other hand, $[G_i,G_j]=1$ whenever the indices $i$ and $j$ satisfy the conditions that $|G/N|<p_i<p_j$ and $P_j$ is non-cyclic. The conclusion is that if $j\in I$ and $[G_i,G_j]\not=1$ for some $i<j$, then $p_i\leq|G/N|$. It follows that $[L,G_j]\not=1$ for every $j\in I$.

Recall that $K$ has been taken as an open normal subgroup of $L$ which  commutes with $G_i$ whenever $p_i>|G/N|$ and $P_i$ is non-cyclic. Write $L=\langle l_1,\dots,l_t,K\rangle$ for some $l_{1},\ldots,l_{t}\in L$.  For every $i\in I$ there exists an index $j\in\{1,\ldots,t\}$, depending on $i$, such that $l_{j}$ acts nontrivially on $P_{i}$, and so  $G_i\cap\gamma_3(G)=[G_i,l_j]$. Therefore for some $k\in \{1,\ldots,t\}$ the set $$E=\{i\in I \mid  [G_i,l_k]\neq1\}$$ must be  infinite. Note that for each $i\in E$, every element in $[G_i,l_k]$ can be written as a commutator $[g_i,l_k]$, where $g_i\in[G_i,l_k]$. In fact, for each $i\in E$ and for any non-commuting elements $a_i,b_i\in C_{G_i}(H_i)$ the subgroup $\langle a_i,b_i\rangle$ commutes with the subgroup $\langle g_{i},l_k\rangle$. Indeed, we have $G_i=C_{G_i}(H_i)\times[G_i,H_i]$. Both elements $a_{i}$ and $b_{i}$ are in $C_{G_{i}}(H_{i})$ and so they commute with $l_{k}\in L$.  Moreover $a_i$ and $b_i$ commute with $g_i$ because $g_i\in[G_i,H_i]$. Therefore by Lemma \ref{commutator} the product $[a_i,b_i][g_{i},l_k]$ is again  a commutator. We are ready to prove the following claim. 
\begin{claim}\label{ultimoclaim}
Every element of the Cartesian product $\prod_{i\in E}P_i$ is a commutator.
\end{claim}
Let $x=\prod_{i\in E} x_{i}$, where $x_{i}$ belongs to $P_{i}$. For any $i$ the element $x_{i}$ can be written as $[a_{i},b_{i}][g_{i},l_{k}]$ for some $g_{i}\in[G_{i},l_{k}]$ and some non-commuting elements $a_i,b_i\in C_{G_i}(H_i)$. Moreover for any $i_{1},i_{2}\in E$, the subgroup  $\langle a_{i_{1}},b_{i_{1}}\rangle$ commutes with the subgroups $\langle a_{i_{2}},b_{i_{2}}\rangle$, $\langle g_{i_{1}},l_{k}\rangle$ and $\langle g_{i_{2}},l_{k}\rangle$. Hence  we have $[g_{i_{1}},l_{k}][g_{i_{2}},l_{k}]=[g_{i_{1}}g_{i_{2}},l_{k}]$. Write $$g_0=\prod_{i\in E}g_{i},\ a_0=\prod_{i\in E}a_{i}, \text{ and } b_0=\prod_{i\in E}b_{i}.$$
Then we have $$x=[a_{i_{1}},b_{i_{1}}][a_{i_{2}},b_{i_{2}}]\cdots[g_{i_{1}}g_{i_{2}}\cdots,l_{k}]=[a_0,b_0][g_0,l_{k}].$$ Since also the subgroup $\langle g_0,l_{k}\rangle$ commutes with the subgroup $\langle a_0,b_0\rangle$, the element $x$ can be written as $[a_0g_0,b_0 l_{k}]$, and so it is a commutator, as claimed. The proof of Claim \ref{ultimoclaim} is complete. 

Since every element in $\prod_{i\in E}P_i$ is a commutator, we conclude that $\prod_{i\in E}P_i$ is virtually procyclic. Obviously we now have a contradiction since $P_i$ is not cyclic whenever $i\in E$.  This concludes the proof of Claim \ref{claimPi}. By Lemma \ref{cyclicprod} $G'$ is virtually procyclic. Lemma \ref{virtprocyclic} implies that $G'$ is finite-by-procyclic. The proof of the theorem is now complete. 
\end{proof}

\section{Main result} 

The proof of Theorem \ref{mainth} will now be fairly easy.

\begin{proof}[Proof of Theorem \ref{mainth}] We wish to prove that if $G$ is a profinite group, then $G'$ is finite-by-procyclic if and only if the set of all commutators of $G$ is contained in a union of countably many procyclic subgroups. 

So assume that the set of all commutators in $G$ is contained in a union of countably many procyclic subgroups. It follows from Theorem 7 of \cite{AS1} that $G'$ is virtually soluble. Of course, this implies that $G$ is virtually soluble. Choose a normal soluble open subgroup $N$ in $G$. Let $d$ be the derived length of $N$ and argue by induction on $d$. If $d\leq 1$, then the result follows from Theorem \ref{virtabeliancase}. Hence we can assume that $d\geq 2$. Let $M$ be the last nontrivial term of the derived series of $N$. By induction $G'/M$ is finite-by-procyclic. Let $T$ be the metabelian term of the derived series of $N$. Then $M=T'$ and so, by Theorem \ref{meta}, $M$ is finite-by-procyclic. Lemma \ref{normal-characteristic} now tells us that $M$ has a finite characteristic subgroup $M_{0}$ such that $M/M_{0}$ is procyclic. We can pass to the quotient $G/M_{0}$ and without loss of generality assume that $M$ is procyclic. Then by Lemma \ref{autocyclic} we have $M\leq Z(G')$. Recall that $G'/M$ is finite-by-procyclic. Applying Lemma \ref{normal-characteristic} to the quotient group $G/M$, we obtain that $G'$ contains a characteristic subgroup $K$ such that $K/M$ is finite and $G'/K$ is procyclic. Since $M\leq Z(G')$, by Schur's theorem $K'$ is finite. We pass to the quotient $G/K'$ and without loss of generality assume that $K$ is abelian. So $K$ is an abelian virtually procyclic subgroup. In view of Remark \ref{remark2} we can assume that $K$ is procyclic. Then, again by Lemma \ref{autocyclic}, $K\leq Z(G')$. Since $G'/K$ is procyclic, it follows that $G$ is metabelian and the result follows from Theorem \ref{meta}. 

Thus, we proved that if the set of all commutators in $G$ is contained in a union of countably many procyclic subgroups, then $G'$ is finite-by-procyclic. The converse is immediate from Proposition \ref{cove}. The proof is now complete.
\end{proof}

\end{document}